\documentclass[12pt,a4paper,twoside]{article}
\usepackage[T1]{fontenc}
\usepackage[english]{babel}
\usepackage{fancyhdr}
\usepackage{indentfirst} 
\usepackage{newlfont}
\usepackage{verbatim}
\usepackage{enumerate}
\usepackage{caption}
\usepackage{hyperref}
\usepackage{pdflscape}
\usepackage[all,cmtip]{xy}
\usepackage{amsthm}  
\usepackage{amssymb}  
\usepackage{amscd}
\usepackage{amsmath} 
\usepackage{latexsym} 
\usepackage{amsfonts}
\usepackage{amsbsy}   
\usepackage{mathrsfs} 
\usepackage[usenames,dvipsnames]{color}
\usepackage{graphicx}
\usepackage{color}

\newtheorem{teo}{Theorem}
\newtheorem{lemma}{Lemma}

\theoremstyle{definition}
\newtheorem{remark}{Remark}
\newcommand{\N}{\mathbb N}
\newcommand{\Q}{\mathbb Q}

\newcommand{\GL}{\textup{GL}}
\renewcommand{\tau}{\Delta}

\newtheorem*{con}{Conjecture}

\begin{document}

\selectlanguage{english}

\title{The complexity of orientable graph manifolds}

\author{Alessia Cattabriga - Michele Mulazzani}


\maketitle

\begin{abstract}
We give an upper bound for the Matveev complexity of the whole class of closed connected orientable prime graph manifolds that is sharp for all 14502  graph manifolds of the Recognizer catalogue (available at  \texttt{http://matlas.math.csu.ru/?page=search}).\\

\end{abstract}

\section{Introduction}

Graph manifolds have been introduced and classified by Waldhausen in \cite{Wa} and \cite{Wa2}. They are defined as compact 3-manifolds  obtained by gluing Seifert fibre spaces  along toric boundary components; so they can be described using  labeled digraphs, as it will be explained in the next section.

Matveev in  \cite{M1} (see also \cite{M2}) introduced the notion of complexity for compact 3-dimensional manifolds,  as a way  to   measure how ``complicated''   a manifold is.  Indeed, for closed irreducible and $\mathbb P^2$-irreducible manifolds the  complexity coincides with the minimum number of tetrahedra needed to construct the manifold, with the only exceptions  of $S^3$, $\mathbb{RP}^3$ and $L(3,1)$, all having complexity zero.  
Moreover, complexity  is additive under connected sum and it is finite-to-one in the closed irreducible case. The last property has been used in order to construct a census of manifolds according to increasing complexity: for the orientable case, up to complexity 12, in the  Recognizer catalogue\footnote{See \texttt{http://matlas.math.csu.ru/?page=search}.}  and for the non-orientable  case, up to complexity 11, in the Regina catalogue\footnote{See \texttt{https://regina-normal.github.io}.} . 

Upper bounds for the complexity of infinite families of 3-manifolds are given in \cite{M0} for lens spaces, in \cite{MP2} for closed orientable Seifert fibre spaces and for orientable torus bundles over the circle, in \cite{FW} for orientable Sefert fibre space with boundary and in \cite{CMMN} for non-orientable compact Seifert fibre spaces. All the previous upper bounds are sharp for  manifolds contained in the above cited catalogues. Furthermore, in \cite{JRT} and \cite{JRT2} it has been proved that the upper bound given in \cite{M0} is sharp for two infinite families of lens spaces.  Very little is known for the complexity of graph manifolds: in \cite{F} and \cite{FO}  upper bounds are given only for the case of graph manifolds obtained by gluing along the boundary two or three Seifert fibre spaces with disk base space and at most two exceptional fibres.

The main goal of this paper is to furnish a potentially sharp upper bound for the complexity of  all  closed connected orientable prime graph manifolds different from  Seifert fibre spaces and  orientable torus bundles over the circle.   It is worth noting that the upper bounds given in Theorems~\ref{regular},  \ref{tree} and \ref{generale} are sharp for all 14502  manifolds of this type included  in the Recognizer catalogue.

The organization  of the paper is the following.  In Section \ref{graph} we recall the definition of graph manifolds and their representation via  decomposition graphs, while in Section \ref{teorema} we state  our results.

We end this section by recalling the definition of complexity and skeletons.

A polyhedron $P$ is said to be \textit{almost simple} if
the link of each point $x\in P$ can be embedded into $K_4$, the
complete graph with four vertices. In particular, the polyhedron is called \textit{simple} if the link is homeomorphic  to  either a circle, or  a circle with a diameter, or  $K_4$. A \textit{true vertex} of an
(almost) simple polyhedron $P$ is  a point $x\in P$ whose link is
homeomorphic to $K_4$.  A \textit{spine} of a closed  connected 3-manifold 
is a  polyhedron $P$ embedded in $M$ such that  $M\setminus P\cong B^3$,  being $B^3$ an open 3-ball. The \textit{complexity} $c(M)$ of $M$ is the minimum number of true vertices among all almost simple spines of $M$.

We will construct a spine for a given graph manifold by gluing skeletons of its Seifert pieces. 
Consider a    compact connected  3-manifold $M$  whose  boundary either is empty or  consists of tori. Following  \cite{MP} and \cite{MP2}, a \textit{skeleton of}  $M$ is a  sub-polyhedron $P$ of $M$ such that (i) $P\cup\partial M$ is simple, (ii) $M\setminus (P\cup \partial M)\cong B^3$, (iii) for any component  $T^2$ of $\partial M$ the intersection $T^2\cap P$ is  a non-trivial theta graph\footnote{A  \textit{non-trivial} theta graph $\theta$ on a torus  $T^2$ is  a subset of $T^2$  homeomorphic to the theta graph (i.e., the graph with 2 vertices and 3 edges joining them),   such that $T^2\setminus \theta$ is an open disk.}. Note that if $M$  is closed then $P$ is a spine of $M$. 
Given two  manifolds  $M_1$ and $M_2$ as above  with non-empty boundary, let $P_i$ be a skeleton of $M_i$, for $i=1,2$. Take two components $T_1\subseteq \partial M_1$ and $T_2\subseteq\partial M_2$ such that $P_i\cap T_i=\theta_i$ and consider an  homeomorphism $\varphi:  (T_1,\theta_1)\to  (T_2,\theta_2)$. Then $P_1\cup_{\varphi}P_2$ is a skeleton for $M_1\cup_{\varphi}M_2$: we call this operation, as well as the manifold $M_1\cup_{\varphi}M_2$, an \textit{assembling} of $M_1$ and $M_2$.

\section{Graph manifolds and their  combinatorial description}

\label{graph}

Let us start by fixing some notations for Seifert fibre spaces. 
We will consider only orientable compact connected Seifert fibre spaces with non-empty boundary, described as  $S=(g,d,(p_1,q_1),\ldots,(p_{r},q_{r}),b)$  where: $g\in\mathbb Z$ coincides with the genus 
of the base space if it is orientable and with the opposite if it is non-orientable,  $d>0$ is the number 
of boundary components of $S$, $(p_j,q_j)$ are lexicographically ordered  pairs of coprime integers such 
that $0<q_j<p_j$, for $j=1,\ldots, r$, describing the type of the exceptional fibres of $S$  and  $b\in\mathbb Z$ 
 can be considered as a (non-exceptional) fibre of type $(1,b)$. 

Up to fibre-preserving homeomorphism, we can assume  (see \cite{FM}) that the Seifert pieces appearing in a graph manifold 
belong to the set $\mathcal S$ of the oriented compact connected Seifert fibre spaces with non-empty 
boundary that are different from  fibred solid tori  and from the fibred spaces $S^1\times S^1\times I$ 
and $N\tilde\times S^1$ (i.e., the orientable circle bundle over the Moebius strip $N$, which will be considered 
with the alternative Seifert fibre structure $(0,1,(2,1),(2,1),b)$).

A Seifert fibre  space $S=(g,d,(p_1,q_1),\ldots,(p_{r},q_{r}),b)\in\mathcal S$, with   base space $B=p(S)$, is equipped with  coordinate systems on the toric boundary components, as follows. 
Let $\widehat S$ be the closed Seifert fibre space $\widehat S=(g,0,(p_1,q_1),\ldots,(p_{r},q_{r}),b)$, with base space \hbox{$\widehat B=\widehat p(\widehat S)$.} Moreover, let $\widehat s:\widehat B'\to \widehat S'$ be the section naturally associated to $\widehat S$, up to isotopy, where $\widehat B'=\widehat p(\widehat S')$ and $\widehat S'$ is obtained from $\widehat S$ by removing $r+1$ open fibred solid tori (with disjoint closures)  which are regular neighborhoods of the exceptional fibres of $\widehat S$ and of the fibre of type $(1,b)$. The manifold $S$ is obtained from $\widehat S$ by removing $d$ open fibred solid tori which are regular neighborhoods of $d$ regular fibres of type $(1,0)$ of $\widehat S$. Therefore, $p=\widehat p_{|_{S}}$ and $\partial(B)\subset \widehat B'$. 
The section $\widehat s$ induces a coordinate system (i.e., a positive basis of the first ho\-mo\-lo\-gy group) $(\mu_h,\lambda_h)$ \label{coordinates} on the $h$-th toric boundary component $T_h$ of $S$, for each $h=1,\ldots, d$, such that: $\mu_h$ is the image under $\widehat s$ of a   boundary component $c'_h$ of $B$ and  $\lambda_h$ is a fibre of $S$. Moreover, $\mu_h$ and $\lambda_h$ are oriented in such a way that their intersection index is 1, with the orientation of $T_h$ induced by the one of $S$. When $B$ is orientable, we require that any $\mu_h$ is oriented according to the orientation of $c'_h$ induced by a fixed orientation on $B$.

Consider a finite connected  non-trivial digraph $G=(V,E,\iota)$, where
\hbox{$V=\{v_i\mid i\in I\}$} is the set of vertices and $E=\{e_j\mid j\in J\}$ is the set of oriented edges of $G$, with incidence structure $\iota:E\to V\times V$ given by $\iota(e_j)=(v_{i'_j},v_{i''_j})$. Let 
$$H=\left(\begin{array}{cc} 0 & 1\\ 1 & 0\end{array}\right) \quad \textup{and}\quad  U=\left(\begin{array}{cc} 1 & 0\\ 1 & 1\end{array}\right),$$
then associate:
\begin{itemize}
 \item to each vertex $v_i\in V$ having degree $d_i$ a Seifert fibre space \linebreak$S_i=(g_i,d_i,(p_1,q_1),\ldots,(p_{r_i},q_{r_i}),b_i)\in \mathcal S$ (i.e., the degree of $v_i$ is equal to the number of components of $\partial S_i$);
 \item  to each edge $e_j\in E$ a matrix $A_j=\left(\begin{array}{cc} \alpha_j & \beta_j\\ \gamma_j & \delta_j\end{array}\right)\in \GL^-_2(\mathbb Z)$ such that $\beta_j\neq 0$ and $0\le \epsilon_j\alpha_j,\epsilon_j\delta_j<|\beta_j|$, where $\epsilon_j=\beta_j/|\beta_j|$. We call \textit{normalized} a matrix of $\GL_2^-(\mathbb Z)$ satisfying these conditions. Moreover:
 \begin{itemize}
 
 \item[(i)]  $A_j\ne \pm H$ when  either $S_{i'_j}$ or  $S_{i''_j}$ is the space $(0,1,(2,1),(2,1),-1)$;
 \item[(ii)]  when $|V|=2$, $|E|=1$ and $S_1=(0,1,(2,1),(2,1),b_1)$, \linebreak $S_2=(0,1,(2,1),(2,1),b_2)$,
 \begin{itemize}
 \item[(a)] if $A=\pm H$ then $(b_1,b_2)\ne (0,0),(-2,-2)$;
 \item[(b)] if $A=\pm \left(\begin{array}{cc} 1 & \beta\\ 1 & \beta-1\end{array}\right)$, with $\beta>1$, then  $(b_1,b_2)\ne (-1,-2)$;
 \item[(c)] if $A=\pm \left(\begin{array}{cc} \beta-1 & \beta\\ 1 & 1\end{array}\right)$, with $\beta>1$, then  $(b_1,b_2)\ne (0,-1)$.
\end{itemize}
\end{itemize}
\end{itemize}

The graph manifold $M$ associated to the above data is obtained by gluing, for each  edge $e_j\in E$  with $\iota(e_j)=(v_{i'_j},v_{i''_j})$, a toric boundary component of $S_{i'_j}$ with one of  $S_{i''_j}$, using the homeomorphism represented by $A_j$ with respect to the fixed coordinate systems on the tori.\footnote{If $\beta_j=0$ for some $A_j$, then the gluing map sends a fiber of $S_{i'_j}$ into a fiber of $S_{i''_j}$. This implies that the decomposition of $M$ in Seifert pieces is not minimal with respect to the number of cutting tori. For the same reason conditions (i) and (ii) hold (see \cite[p. 279]{FM}). Observe that we can assume $\beta_j>0$, for any $j$ such that   $e_j$ belongs to a fixed spanning tree of $G$. } Clearly, $M$ is a closed, orientable and connected graph manifold. On the other hand, each  closed connected  orientable prime graph manifold different from a Seifert fibre space and an orientable torus bundle over the circle  can be obtained in this way (see \cite[\S 11]{FM}). We will call $G$ a \textit{decomposition graph} of $M$.

If $G'=(V, E')$ is a spanning  subgraph of a decomposition graph $G$ we denote by $M_{G'}$ the graph manifold (with boundary if $G'\ne G$) obtained by performing only the attachments corresponding to the elements  of $E'$.

\begin{remark}
\label{regular_matrices}
There is no restriction in assuming that all  matrices associated to the edges of a decomposition graph are normalized: this is because of  the  following two operations that do not change the resulting graph manifold (see \cite[\S 11]{FM} and \cite{VMF}):
 \begin{itemize}
  \item[1)] replacement of the matrix $A_j$ with $A_jU^k$ and of the parameter $b_{i'_j}$ of the  Seifert space  $S_{i'_j}$ with the parameter $b_{i'_j}+k$;
  \item[2)] replacement of the matrix $A_j$ with $U^kA_j$ and of the parameter $b_{i''_j}$ of the Seifert space $S_{i''_j}$  with $b_{i''_j}-k$.
 \end{itemize}
Indeed, given  
 a matrix  $A=\left(\begin{array}{cc} \alpha & \beta\\ \gamma & \delta\end{array}\right)\in\GL^-_2(\mathbb Z)$,   let  $k=-\lfloor\frac{\alpha}{\beta}\rfloor$ and \hbox{$h=-\lfloor\frac{\delta}{\beta}\rfloor$} where $\lfloor x\rfloor$ denotes the floor   of $x$. Then   the matrix $$A'=U^h A   U^k=\left(\begin{array}{cc}\alpha+k\beta&\beta\\ \gamma +h\alpha+k\delta + kh\beta&\delta +h\beta\end{array}\right)$$  is normalized. Note that for a normalized   matrix $A=\left(\begin{array}{cc} \alpha & \beta\\ \gamma & \delta\end{array}\right)$ the fol\-lo\-wing properties hold:
 \begin{itemize}
 \item $\beta\gamma>0$;
 \item if $\beta=\pm 1$ then $A=\beta H$;
  \item if $A\ne \pm H$ then $\beta/\delta>0$.
  
 \end{itemize}
Moreover $A\in \GL^-_2(\mathbb Z)$ is normalized if and only if $-A$ is normalized.
 \end{remark}

\subsection{Theta graphs and Farey triangulation}

Consider the upper half-plane model  of the hyperbolic plane $\mathbb H^2$ and let  $\mathbb F$  be the  ideal Farey triangulation (see \cite{Bo}).  The  vertices  of $\mathbb F$ coincide with the points of $\mathbb Q\cup\{\infty\}\subset \mathbb R\cup \{\infty\}=\partial \mathbb H^2$ and the edges of $\mathbb F$ are geodesics in $\mathbb H^2$ with endpoints the pairs $a/b$, $c/d$ such that $ad-bc=\pm1$,  with $\pm 1/0=\infty$.  Let  $\tau_{\frac ab,\frac cd,\frac ef}$ be the triangle of the Farey triangulation with vertices $a/b,c/d,e/f\in\mathbb Q\cup\{\infty\}$ and  set $\tau_+=\tau_{\infty,0,1}$,  $\tau_-=\tau_{\infty,0,-1}$.

Let $T^2$ be a  torus, it is a well-known fact that the vertex set of $\mathbb F$ is in bijection with the set of slopes (i.e., isotopy classes of non-contractible simple closed curves) on $T^2$ via $a/b\leftrightarrow a\mu+b\lambda$, where  $(\mu,\lambda)$ is  a  fixed basis of $H_1(T^2)$.  This bijection induces a bijection between the set of triangles of the Farey triangulation  and the set $\Theta(T^2)$ of  non-trivial theta graphs  on $T^2$, considered up to isotopy. Indeed, given  $\theta\in\Theta(T^2)$, consider the three slopes  $l_1,l_2,l_3$ on $T^2$ formed by the pairs of edges of $\theta$. The triangle associated to  $\theta$ is  $\tau_{l_1,l_2,l_3}$. Note that this bijection is well defined since  the intersection index of $l_i$ and $l_j$, with $i\ne j$, is always $\pm1$.

 The graph $\mathbb F^*$  dual to $\mathbb F$ is an infinite tree. Given two triangles $\tau$ and $\tau'$ in $\mathbb F$ the distance $d(\tau,\tau')$ between them is the number of edges of the unique simple path joining the vertices $v_{\tau}$ and $v_{\tau'}$ corresponding to $\tau$ and $\tau'$ in $\mathbb F^*$, respectively.   Given two theta graphs $\theta, \theta'\in \Theta(T^2)$ it is possible to pass from one to the other by a sequence of flip moves (see  Figure \ref{flip2}):  the distance on the set of triangles of the Farey triangulation induces a distance  on $\Theta(T^2)$ such that  $d(\theta,\theta')$ turns out to be  the minimal number of flips necessary to pass from $\theta$ to $\theta'$ (see \cite{MP2}).

\begin{figure}[h!]                      
\begin{center}                         
\includegraphics[width=6cm]{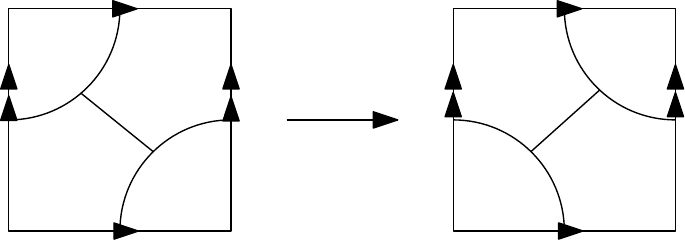}
\caption[legenda elenco figure]{Two theta graphs connected by a flip move.}\label{flip2}
\end{center}
\end{figure}

The group $\GL_2(\mathbb Z)$ acts on $\mathbb  H^2$ as isometries  and $\mathbb F$ is invariant under this action: given a triangle $\tau_{\frac ab,\frac cd,\frac ef}\in\mathbb F$ if we associate to it the matrix 
$$\left(\begin{array}{ccc} a & c& e\\ b& d & f\end{array}\right),$$
then  the group $\GL_2(\mathbb Z)$ acts on the set of triangles of the Farey triangulation by left multiplication.  

The \textit{complexity} $c_A$ of a matrix $A\in \GL_2(\mathbb Z)$ is defined as
$$c_A=\textup{min}\left\{d(A\tau_-,\tau_-), d(A\tau_-,\tau_+), d(A\tau_+,\tau_-), d(A\tau_+,\tau_+)\right\}.$$

Now we state a result about the complexity of nor\-ma\-li\-zed matrices. Let  $S:\Q^+\to\N$ defined by $S(a/b)=a_1+\cdots+a_k$, where  
$$\frac{a}{b}= a_1+\cfrac{1}{\ddots +\cfrac{1}{a_{k-1} +\cfrac{1}{a_k}}},
 $$
is the expansion of the positive rational number $a/b$ as a continued fraction, with $a_1,\ldots,a_k>0$.

\begin{lemma} \label{matrice} Let $A=\left(\begin{array}{cc} \alpha & \beta\\ \gamma & \delta\end{array}\right)\in \GL^-_2(\mathbb Z)$ be a normalized matrix.  
\begin{itemize}
 \item If $A=\pm H$ then $c_A=d(A\tau_-, \tau_-)=d(A\tau_+, \tau_+)=0$.
  \item If $A\ne \pm H$ then $c_A=d(A\tau_-, \tau_+)=S(\beta/\delta)-1$.
\end{itemize}
\end{lemma}

\begin{proof}
The first statement is straightforward since $\pm H\tau_{\pm}=\tau_{\pm}$. To prove the second one let $A=\left(\begin{array}{cc} \alpha & \beta\\ \gamma & \delta\end{array}\right)\ne\pm H$ and so $|\beta|>1$ (see Remark \ref{regular_matrices}). Denote by $\mathcal D_{\beta/\delta}$ the set of triangles of the Farey triangulation having a vertex in $\beta/\delta$.  By \cite[Lemma 4.3]{MP} we have    $\textup{min}\{d(\tau,\tau_+)\mid \tau\in\mathcal D_{\beta/\delta}\}=S(\beta/\delta) -1$. If $A$ is normalized then: $$0<\frac{\alpha}{\gamma}<\frac{\beta+\alpha}{\gamma+\delta}<\frac{\beta}{\delta}<\frac{\beta-\alpha}{\delta-\gamma}.$$ Indeed, since $\alpha\delta-\beta\gamma=-1$, we have:
$$\frac{\alpha}{\gamma}<\frac{\alpha}{\gamma}+\frac{1}{\delta\gamma}=\frac{\alpha\delta+1}{\delta\gamma}=\frac{\beta}{\delta}.$$
So,
$$
 0<\frac{\alpha}{\gamma}=\frac{\frac{\alpha}{\gamma}+\frac{\alpha}{\delta}}{1+\frac{\gamma}{\delta}}<
 \frac{\frac{\beta}{\delta}+\frac{\alpha}{\delta}}{1+\frac{\gamma}{\delta}}=\frac{\beta+\alpha}{\delta+\gamma}=\frac{\frac{\beta}{\gamma}+\frac{\alpha}{\gamma}}{\frac{\delta}{\gamma}+1}<\frac{\frac{\beta}{\gamma}+\frac{\beta}{\delta}}{\frac{\delta}{\gamma}+1}=\frac{\beta}{\delta}=\frac{\frac{\beta}{\gamma}-\frac{\beta}{\delta}}{\frac{\delta}{\gamma}-1}<\frac{\frac{\beta}{\gamma}-\frac{\alpha}{\gamma}}{\frac{\delta}{\gamma}-1}=\frac{\beta-\alpha}{\delta-\gamma},$$
where we suppose $\delta-\gamma \ne 0$, otherwise the last inequality is straightforward.
Clearly $A\tau_-=\tau_{\frac{\alpha}{\gamma},\frac{\beta}{\delta},\frac{\beta-\alpha}{\gamma-\delta}}$, $A\tau_+=\tau_{\frac{\alpha}{\gamma},\frac{\beta}{\delta},\frac{\beta+\alpha}{\gamma+\delta}}$ and the relative position of the triangles is represented in  Figure \ref{farey}, where for convenience we use the Poincar\'e  disk model of $\mathbb H^2$.  All  triangles of $\mathcal D_{\beta/\delta}$ different from  $A\tau_-$ and $A\tau_+$ are contained in the two hyperbolic half-planes  depicted in gray. As a consequence,  we have $\textup{min}\{d(\tau,\tau_+)\mid \tau\in\mathcal D_{\beta/\delta}\}=d\left(A\tau_-,\tau_+\right)$. Since  the path  in $\mathbb F^*$ going from $v_{A\tau_{+}}$ to $v_{\tau_-}$  contains $v_{A\tau_{-}}$ and  $v_{\tau_+}$, we have $c_A=d(A\tau_{-},\tau_+)=S(\beta/\delta)-1$.   
\end{proof}

We end this section by introducing a notation to distinguish edges corresponding to $\pm H$ from the others. Given a  decomposition graph $(V,E,\iota)$ with $V=\{v_i\mid i\in I\}$ and $E=\{e_j\mid j\in J\}$, let  $E'=\{e_j\mid A_j=\pm H\}$, $E''=E-E'$ and $J'=\{j\in J\mid e_j\in E'\}$, $J''=J-J'$. Moreover, set $d^+_i=|\{j\in J''\mid i'_j=i\}|$ and $d^-_i=|\{j\in J''\mid i''_j=i\}|$. Observe that, if $d^0_i=|\{j\in J'\mid i'_j=i\}|+|\{j\in J'\mid i''_j=i\}|$ then $d_i=d^+_i+d^-_i+d^0_i$.

\begin{figure}[h!]                      
\begin{center}                         
\includegraphics[width=8cm]{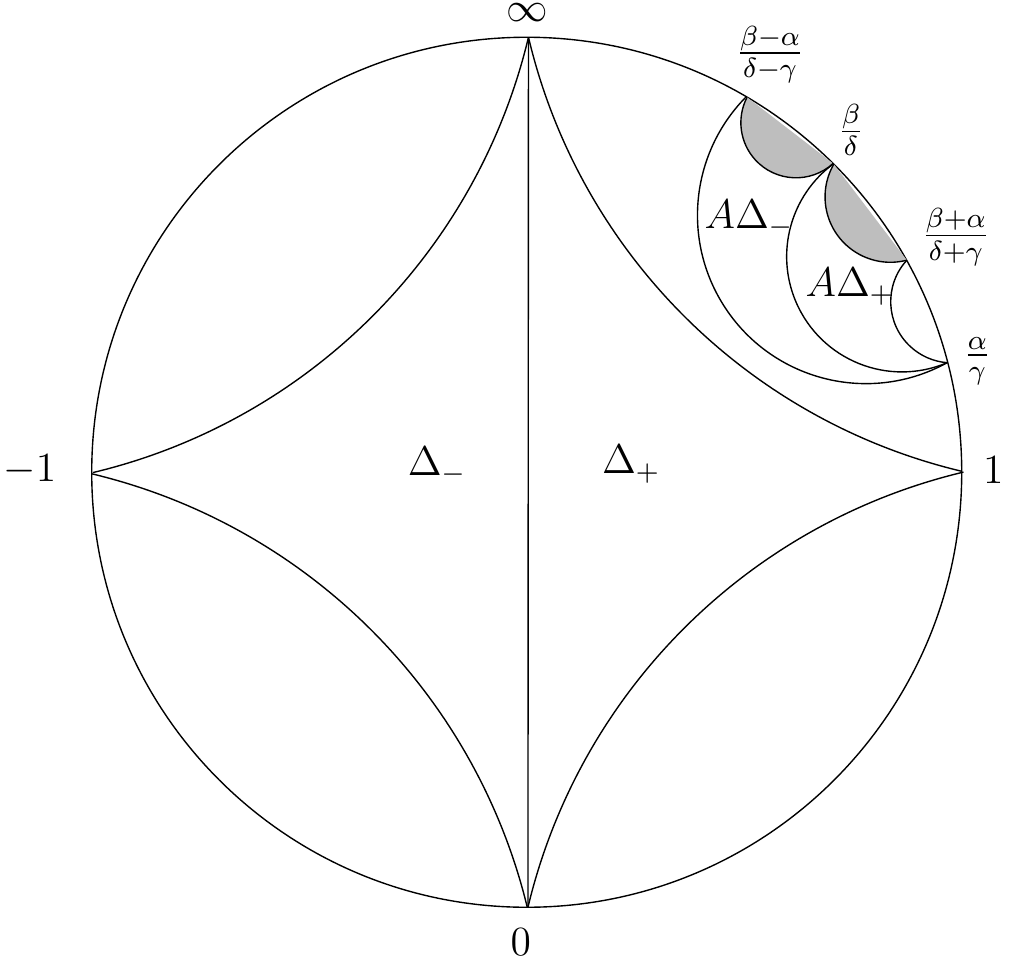}
\caption[legenda elenco figure]{The Farey triangulation in the Poincar\'e disk model.}\label{farey}
\end{center}
\end{figure}

\section{Complexity upper bounds}
\label{teorema}

In this section we provide an upper bound for the complexity of graph manifolds. Before stating the general result (Theorem \ref{generale}), we deal with two special classes of graph manifolds: the first one  (Theorem \ref{regular}) is the class of manifolds having decomposition graphs in which all the edges are associated to matrices different from $\pm H$, that is $J'=\emptyset$, while the second one  (Theorem \ref{tree}) concerns manifolds whose decomposition graphs admit a spanning tree containing all  the edges associated to matrices of type $\pm H$.

In all cases we obtain a spine for a graph manifold starting from skeletons of its Seifert pieces.
The construction of these skeletons  is essentially the one described in  \cite{CMMN}, specialized to our case (i.e., orientable Seifert manifolds) and adapted to take care of the fact that the boundary components of the Seifert pieces will be glued together to obtain a closed graph manifold. Anyway, for the sake of the reader we recall  how to construct a  skeleton  for   the Seifert pieces. The  construction and the number of true vertices of the resulting skeletons depend on some choices: we will discuss in the proof of the theorems  how to fix them in order to minimize the number of true vertices of the  spine.  \\

\textbf{Skeletons of  Seifert pieces}
Let   $S=(g,d,(p_1,q_1),\ldots,(p_r,q_r),b)\in\mathcal S$. Denote with  $S_0=(g,d,(p_1,q_1),\ldots,(p_r,q_r),0)$ and let    $S_0'=(g,d+r+1,0)$   be the space obtained from $S_0$ by removing $r+1$  open fibred solid tori (with disjoint closures) which are regular neighborhoods of the exceptional fibres of $S_0$ and of a regular  fibre of type $(1,0)$ contained in $\textup{int}(S_0)$. Then $S_0\setminus\textup{int}(S'_0)=\Phi_0'\sqcup\Phi_1\sqcup\cdots \sqcup \Phi_r$, where $\Phi_k$ (resp.  $\Phi_0'$)  is a closed solid torus having the $k$-th exceptional fibre (resp. a regular fibre) as core.    Let $p_0:S_0\to B_0$ and $p:S\to B$ be the projection maps and set $s=d+r$. 
Note that if $g\geq 0$  (resp. $g<
0$) then $B'_0=p_0(S'_0)$ is a disk with $2g+s$ orientable (resp. $-g$ non-orientable and $s$ orientable) handles attached. We set  $h=2g$ if $g\geq 0$ and $h=-g$ if $g<0$.

Let $D=p_0(\Phi_0')$ and let $A_0$ be the union of the disjoint arcs properly embedded in $B'_0$ depicted by thick lines in  Figure~\ref{fig_main3}. 
Then $A_0$ is non-empty and  is composed by $h$ edges with both endpoints in $\partial D$  and $s$ edges with an endpoint in $\partial D$ and the other one in a different  component of $\partial B'_0$. By construction, $B'_0\setminus\left(A_0\cup\partial B'_0\right)$ is homeomorphic to an open disk and  the  number of points of $A_0$ belonging to $\partial D$ is at least three, since the conditions on the class $\mathcal S$ ensure that $s+2h> 2$.

\begin{figure}[h!]                      
\begin{center}                         
\includegraphics[width=7cm]{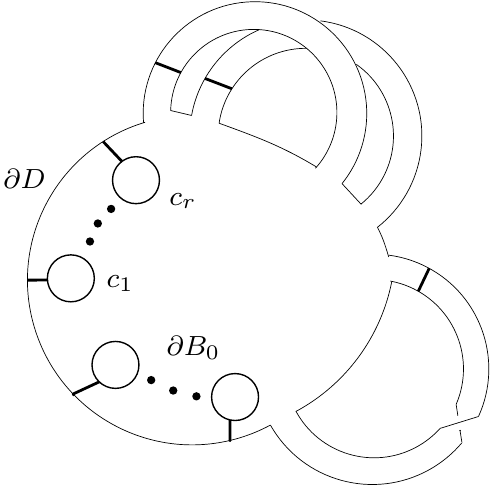}
\caption[legenda elenco figure]{The set $A_0\subset B'_0\setminus\textup{int}(D)$, with $c_k=p_0(\partial \Phi_k)$.}\label{fig_main3}
\end{center}
\end{figure}

Let $s_0':B'_0\to S'_0$ be a section of $p_0$ restricted to $S'_0$. If $b\ne 0$, it is convenient to replace the fibre of type $(1,b)$ with $|b|$ fibres of type $(1,\textup{sign}(b))$. In this way the manifold  $S$ is obtained from  $S_0$   by removing $\vert b\vert$ open   trivially fibred solid tori (with disjoint closures) $\textup{int}(\Phi'_1),\ldots, \textup{int}(\Phi'_{|b|})$, each being a fiber-neighborhood of  regular fibres $\phi_1,\ldots\phi_{|b|}$ contained in $\textup{int}(S_0)$, and by attaching back $|b|$  solid tori $D^2\times S^1$  via  homeomorphisms \hbox{$\psi_l:\partial(D^2\times S^1) \to \partial \Phi'_l$}  such that $\psi_l(\partial D^2\times \{*\})$ is  a curve of type $(1,\textrm{sign}(b))$ on $\partial \Phi'_l$, with respect to a positive  basis  $(\mu_l,\lambda_l)$  of $H_1(\partial\Phi_l')$, where $\mu_l=s_0'(p_0(\partial \Phi'_l))$ and $\lambda_l$ is the fibre over a point  $*'\in p_0(\partial\Phi'_l)$, for $l=1,\ldots,\vert b\vert$. Referring  to Figure \ref{quattro}, it is convenient to take the fibre $\phi_l$ corresponding to an internal point $Q_l$ of $A_0$ and suppose that $p_0(\Phi'_l)$ is a ``small'' disk intersecting the component $\delta_l$ of $A_0$ containing $Q_l$ in an interval and  being disjoint from $\partial B'_0$ and from the other components of $A_0$.  In this way $\delta_l\setminus \textrm{int}(p_0(\Phi'_l))$ is the disjoint union of two arcs $\delta'_l$ and $\delta''_l$. Let $A=A_0\setminus\cup_{l=1}^{\vert b\vert} \textrm{int}(p_0(\Phi'_l))$ and note that $p$ and $p_0$ coincide on $S_0\setminus \cup_{l=1}^{|b|}\textrm{int}(\Phi'_l)$.

\begin{figure}[h!]                      
\begin{center}                         
\includegraphics[width=9cm]{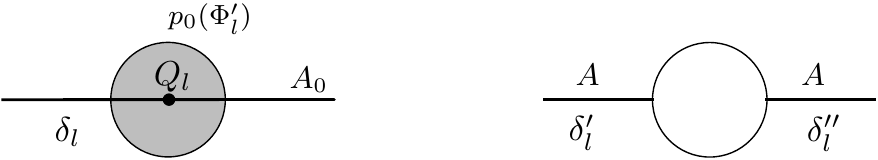}
\caption[legenda elenco figure]{The set $A\subset B'_0$.}\label{quattro}
\end{center}
\end{figure}

Let  $\bar s: \left( B'_0\setminus\left(\cup_{l=1}^{|b|}\textup{int}(p_0(\Phi'_l))\right)\right)\cup D\to S$   be a section of $p_{|_{\textup{Im}(\bar s)}}$ and consider   the polyhedron \hbox{$P=\textup{Im}(\bar s)     \cup p^{-1}(A)\cup \partial \Phi'_0\cup_{l=1}^{\vert b\vert}\left( \partial\Phi_l'\cup_{\psi_l}\left(D^2\times \{*\}\right)\right)\subset S$.}    As represented in the central picture of Figure \ref{freedom}, the set  $\textup{int}(\bar s(A))$ is a collection of quadruple lines in the  polyhedron (the link of each point is homeomorphic to a graph with two vertices and four edges connecting them), and a similar phenomenon occurs for $\bar s(\partial D\setminus A)$.
Therefore we change the polyhedron $P$  performing  ``small'' shifts by moving in parallel the disk $\bar s(D)$ along  the fibration and the components of $p^{-1}(A)$ as depicted  in  the left and right pictures of Figure~\ref{freedom}.  It is convenient to think the shifts of $p^{-1}(A)$ as performed on the components of $A$. Moreover,  the shifts on $\delta'_l$ and $\delta_l''$ can be chosen independently.

\begin{figure}[h!]                      
\begin{center}                         
\includegraphics[width=13.5cm]{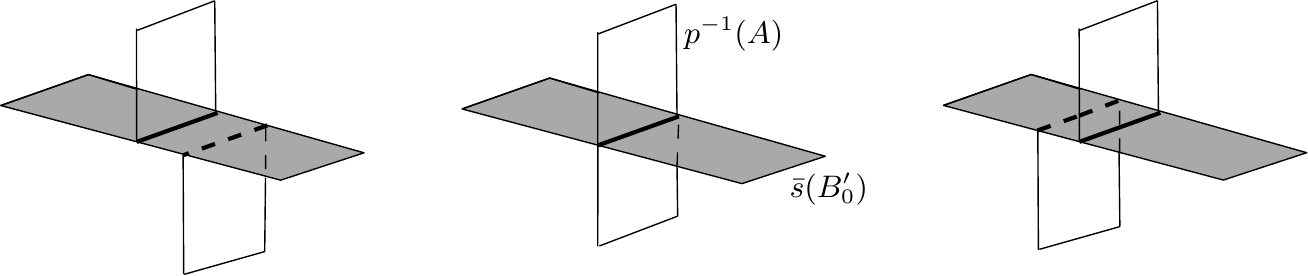}
\caption[legenda elenco figure]{The two possible shifts on a component of $p^{-1}(A)$.}\label{freedom}
\end{center}
\end{figure}

As shown by the pictures, the shift of any component of  $p^{-1}(A)$ may be performed in two different ways that  are not usually equivalent in term of complexity of the final spine. On the contrary, the two possible parallel shifts for $\bar s(D)$ are equivalent as it is evident from Figure~\ref{alette}, which represents the torus $\partial\Phi'_0$.  Let 
$$P'=\bar s\left( B_0'\setminus\left(\cup_{l=1}^{|b|}\textup{int}(p_0(\Phi'_l))\right)\right)\cup D'\cup W' \cup \partial\Phi'_0\cup_{l=1}^{|b|}\left( \partial\Phi'_l\cup_{\psi_l}\left(D^2\times \{*\}\right)\right)$$ 
be the polyhedron obtained from $P$ after the shifts, where $D'$ and $W'$ are the results of the shifts of  $\bar s(D)$ and  $p^{-1}(A)$, respectively.

It is easy to see that $P'\cup\partial S\cup_{k=1}^r \partial \Phi_k$ is simple, $P'$ intersects each component of $\partial S$ and each torus $\partial  \Phi_k$  in a non-trivial theta graph and  the manifold  $S\setminus\left(P'\cup\partial S\cup_{k=1}^r  \Phi_k\right)$ is the disjoint union of $\vert b\vert +2$  open balls. So in order to obtain a skeleton $P''$  for  $S\setminus \left(\cup_{k=1}^r \textup{int}(\Phi_k)\right)$ it is enough to remove a suitable open 2-cell from the torus $T_0=\partial \Phi'_0$ and one from each torus $T_l=\partial\Phi'_l$, for $l=1,\ldots,\vert b\vert$, connecting in this way the balls.

The graph  $\Gamma_l=T_l \cap  \left(\bar s\left( B'_0\setminus \textup{int}(p_0(\Phi'_l))\right)\cup W'\cup \psi_l(\partial D^2\times\{*\})\right)$ (resp. $\Gamma_0=T_0\cap \left(\bar s\left( B'_0 \right)\cup D' \cup W'\right)$)   is cellularly embedded in $T_l$ (resp. $T_0$)  and its  vertices with  degree greater than 2 are true vertices of $P'\cup \partial S\cup_{k=1}^r\partial \Phi_k$:  we will remove the region $R_l$ (resp. $R_0$) of $T_l\setminus \Gamma_l$ (resp. $T_0\setminus \Gamma_0$) having in the boundary the greatest number of vertices of $\Gamma_l$, for $l=1,\ldots,\vert b\vert$  (resp. $\Gamma_0$).

Referring to Figure \ref{alette},  the graph $\Gamma_0$ is composed by two horizontal parallel loops  $\xi=\partial (\bar s(D))$ and $\xi'=\partial D'$, and an arc with both  endpoints on  $\xi$ for each boundary point of $A$ belonging to $\partial D$.  Changing the shift of a component of $A$  has the same effect as performing  a  symmetry along $\xi$ of the correspondent arc(s). A region of $T_0\setminus \Gamma_0$ has 4 or 6 vertices   when the non-horizontal arcs belonging to its boundary are not parallel  or   5  vertices otherwise. So, except the case in which all the arcs are parallel there is always  a region with 6 vertices.

\begin{figure}[h!]                      
\begin{center}                         
\includegraphics[width=10cm]{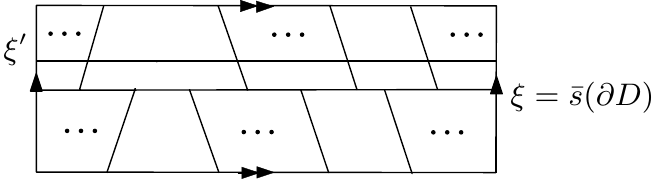}
\caption[legenda elenco figure]{A fragment of the graph $\Gamma_0$ embedded in $\partial \Phi'_0$.}\label{alette}
\end{center}
\end{figure}

When $b\ne 0$, the graph $\Gamma_l$, for $l=1,\ldots, |b|$, is depicted in Figure~\ref{C}  (resp. Figure~\ref{C2})  for a fibre  of type $(1,1)$  (resp. $(1,-1)$), just labeled by  $+$  \hbox{(resp. $-$)} inside the disk.  If we take for $\delta'_l$ and $\delta''_l$ the shifts induced by that of $\delta_l$, then we can choose as region $R_l$ the gray one, containing in its boundary all  vertices  of $\Gamma_l$  belonging to $\partial \Phi_l'$  except one (the thick points in the first two pictures). On the contrary, if one of the two shifts is changed as in the third draw of Figures~\ref{C} and  \ref{C2}, then $R_l$ can be chosen containing in its boundary all the  vertices of  $\Gamma_l$ belonging to $\partial \Phi_l'$.

\begin{figure}[h!]                      
\begin{center}                         
\includegraphics[width=12cm]{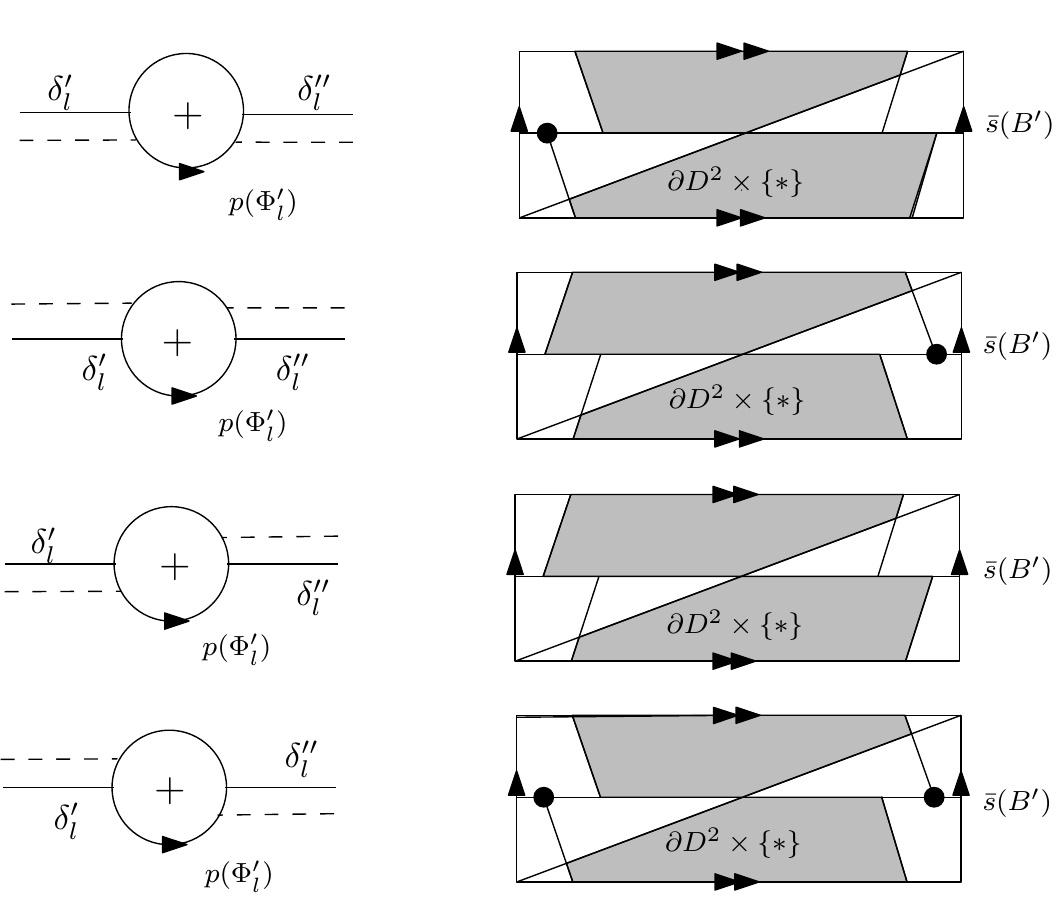}
\caption[legenda elenco figure]{The graph $\Gamma_l$, with $b>0$, embedded in  $T_l=\partial \Phi_l'$ with different choices of the shifts for $\delta'_l$ and $\delta''_l$.}\label{C}
\end{center}
\end{figure}

\begin{figure}[h!]                      
\begin{center}                         
\includegraphics[width=12cm]{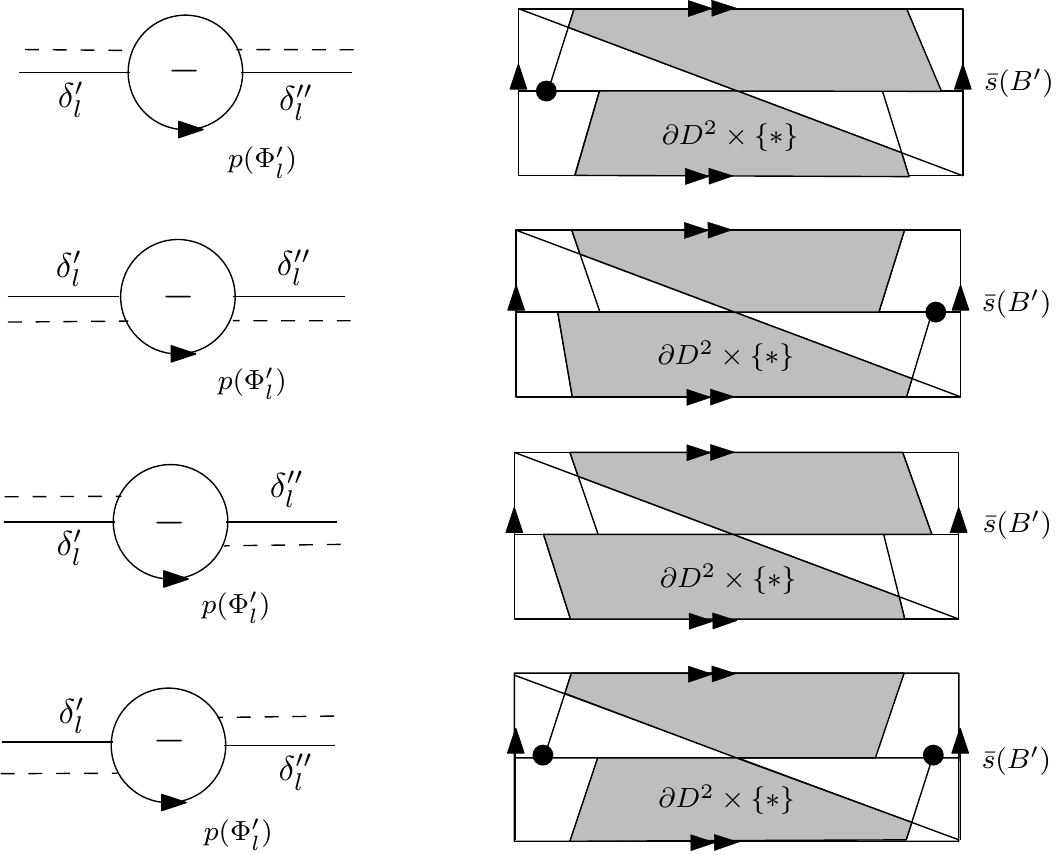}
\caption[legenda elenco figure]{The graph $\Gamma_l$, with $b<0$ embedded in  $T_l=\partial \Phi_l'$ with different choices of the shifts for $\delta'_l$ and $\delta''_l$.}\label{C2}
\end{center}
\end{figure}

We remark that  changing  the shift of a component of $A$   changes the intersection between the corresponding element of $W'$ and $\partial S$ (which is a non-trivial theta graph) by a flip move (see  Figure \ref{flip2}). We denote with $P''$ the skeleton obtained by removing the regions $R_0$ and $R_l$ from $P'$, for $l=1,\ldots,|b|$.

In order to construct a  skeleton for $\Phi_k$, for $k=1,\ldots, r$, consider the skeleton $P_F$ depicted in Figure \ref{flip}: it is a skeleton for $T^2\times [0,1]$ with one true vertex  and such that $\theta_0=P_F\cap (T^2\times \{0\})$ (the graph in the upper face) and $\theta_1=P_F\cap (T^2\times \{1\})$ (the graph in the bottom face) are two theta graphs differing for a flip move. Denote with $\Theta_{p_k/q_k}$ the subset of $\Theta(T^2)$, consisting of the theta graphs containing the slope corresponding to $p_k/q_k\in\mathbb Q\cup\{\infty\}$.  Let $\theta_{p_k/q_k}$ be the theta graph in $\Theta_{p_k/q_k}$ that is closest  to $\theta_+$.  The skeleton $X_k$ for $\Phi_k$ is  obtained by assembling several skeletons of type $P_F$  connecting the theta graph $P''\cap \Phi_k$ to a theta graph which is  one step closer to $\theta_+$ than  $\theta_{p_j/q_j}$,with respect to the distance on $\Theta(T^2)$ (see \cite{FW}).   The number of the required flips is either $S(p_j,q_j)-2$ or $S(p_j,q_j)-1$  depending on the  shift chosen for the corresponding component  of $A$ used in the construction of the skeleton  $P''$. We call the shift  {\it regular} in the first case and  {\it singular}  in the second one (see Figure \ref{A}).

The skeleton $P_{S}$  of $S$ is obtained by assembling $P''$ with $X_k$, via the identity, for $k=1,\ldots,r$.  
\begin{figure}[h!]                      
\begin{center}                         
\includegraphics[width=4cm]{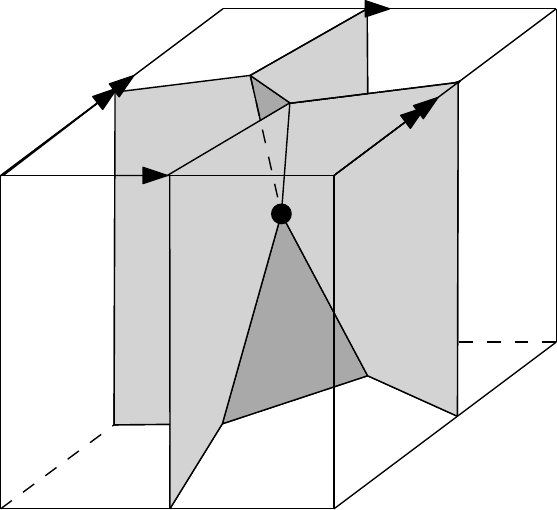}
\caption[legenda elenco figure]{A skeleton for $T^2\times [0,1]$ connecting two theta graphs differing by a flip move.}\label{flip}
\end{center}
\end{figure}

\begin{figure}[h!]                      
\begin{center}                         
\includegraphics[width=5cm]{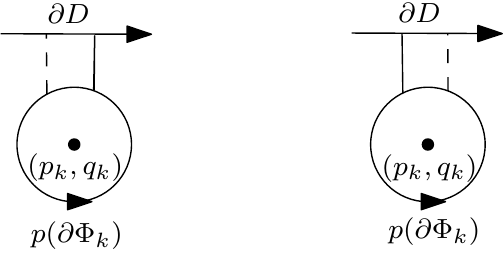}
\caption[legenda elenco figure]{Regular shift (on the left) and singular shift (on the right).}\label{A}
\end{center}
\end{figure}

We end this section by introducing a function that will be useful in the following: let
 $f_{m,M}:\mathbb Z\to\mathbb N$ be defined by 
$$f_{m,M}(b)=\left\{\begin{array}{ll} m-b & \textup{ if }  b<m\\ 0 &\textup{ if }m\leq b\leq M\\ b-M & \textup{ if }  b>M  \end{array}\right.,$$ for $m,M\in \mathbb Z$, $m< M$, $m\leq 1$ and $M\geq-1$ (see the graph in Figure \ref{function}).\\

\begin{figure}[h!]                      
\begin{center}                         
\includegraphics[width=10cm]{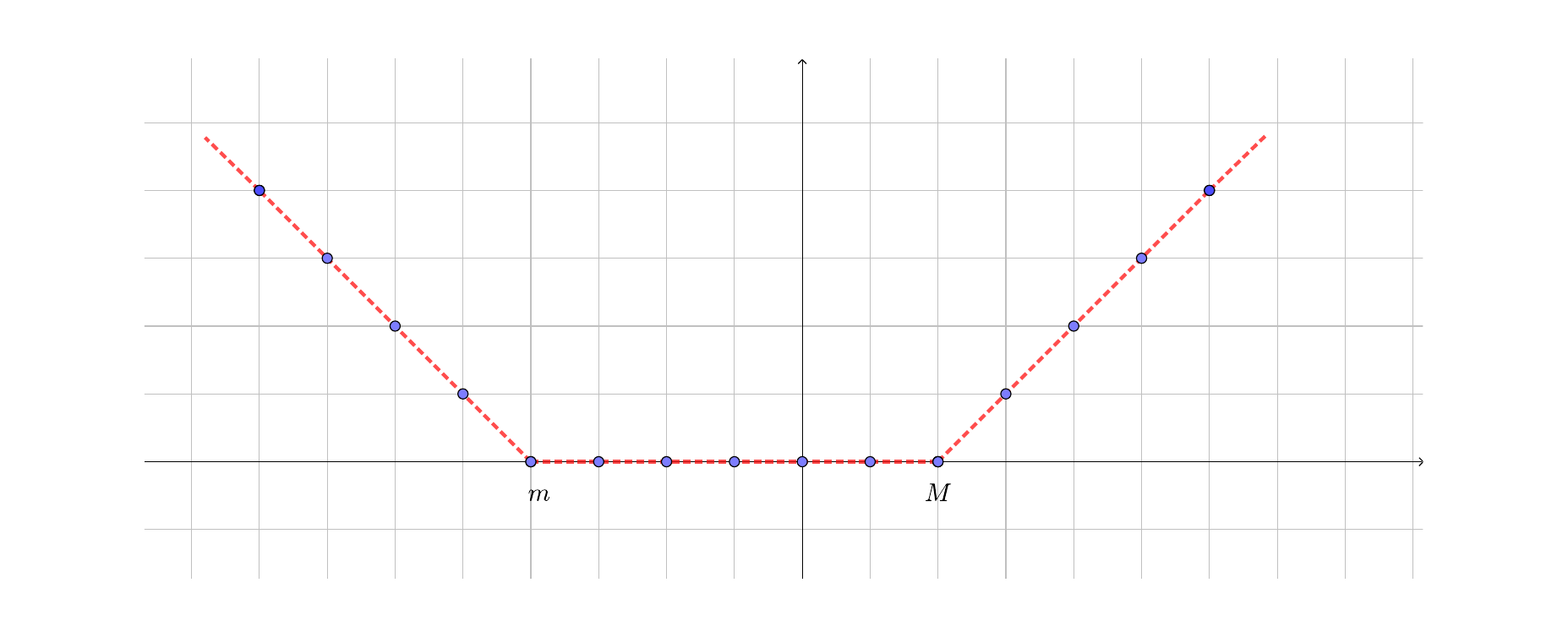}
\caption[legenda elenco figure]{The graph of the function $f_{m,M}$.}\label{function}
\end{center}
\end{figure}

Moreover, let 
${\cal T}_G=\{T=(V,E_T)\mid T \textrm{ is a spanning tree of }G\}$ and for $T\in \mathcal T_G$ set  $J_T=\{j\in J\mid e_j\in E_T\}$, $E'_T=E_T\cap E'$, $E''_T=E_T\setminus E'_T$, $J'_T=\{j\in J\mid e_j\in E'_T\}$, and \hbox{$J''_T=J_T\setminus J'_T$.}

\subsection{The case  $\mathbf{J'=\emptyset}$}

Before stating our first result, we need to discuss how to fix the choices in the construction of the skeleton $P_{S}$ previously described, when the Seifert fibre space $S=(g,d,(p_1,q_1),\ldots,(p_r,q_r),b)$  is  a piece of  a graph manifold having all gluing matrices  different from $\pm H$.  According to the notations introduced at the end of Section  \ref{graph}, we have $d=d^++d^-$ since $J'=\emptyset$. 

\begin{remark}
\label{ottimale}
 Let $\theta_+$ and $\theta_-$ be the theta graphs corresponding, respectively,  to $\tau_+$ and $\tau_-$ in the Farey triangulation. 
  The intersection of each  boundary component of $S$  with the skeleton $P_{S}$ is either $\theta_+$ or $\theta_-$, depending whether the shift of the corresponding component $\delta$ of $A$ has been chosen as depicted in the left or right part of Figure \ref{A2}, respectively.

\begin{figure}[h!]                      
\begin{center}                         
\includegraphics[width=5cm]{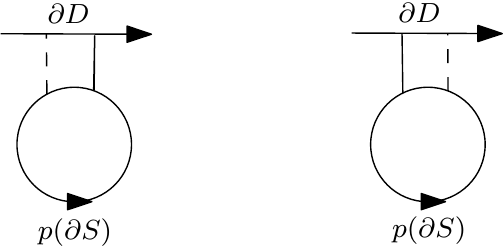}
\caption[legenda elenco figure]{The two possible choices for the shift corresponding to components of $\partial S$.}\label{A2}
\end{center}
\end{figure}

  We always  choose these shifts such  that exactly $d^+$ (resp. $d^-$) components have $\theta_-$ (resp. $\theta_+$) as intersection with $P_{S}$.  Suppose that  $m\leq b\leq M$, where $m=-r-h-d^-+1$ and $M=h+d^+-1$. If $b\leq -1$ we can  choose $p_0(\Phi'_1),\ldots, p_0(\Phi_{|b|}')$ as $|b|$ disks between those marked with $-$  in Figure \ref{negativo}.  In this way:   (i) we can remove a region from $T_0$ containing 6  vertices of $\Gamma_0$, (ii) we can remove from $T_l$ a region   $R_l$  containing in its boundary all  the  vertices of the graph  $\Gamma_l$  (as in the third draw of Figure~\ref{C2}), for each $l=1,\ldots, |b|$,  and (iii) we can take all regular  shifts in the skeletons  $X_k$ corresponding to the exceptional fibres, for $k=1,\ldots, r$.  If $b=0$ we do not have to remove any regular neighborhood $\Phi'_l$ of regular fibres but still (i) and (iii) hold.   An analogue situation happens if $b\geq 1$, but in this case in order to satisfy (i), (ii) and (iii) the  fibres of type $(1,1)$ correspond  to some of the disks marked with $+$ in  Figure \ref{positivo}.  As  a result, when $m\leq b\leq M$ the  polyhedron  $P_S$ has $3(d+r+2h-2)+\sum_{k=1}^{r}(S(p_k/q_k)-2)$ true vertices.

  If  $b<m\leq 0$ (resp. $b>M\geq 0$) (i) and (iii) hold and  there are exactly   $m-b$ (resp. $b-M$) tori in which we remove a region $R_l$ containing in its boundary all  the vertices of $\Gamma_l$  except one (see the first two pictures of Figure \ref{C} and \ref{C2}). Finally, if either $b< m=1$ or $b>M=-1$ then (i) does not hold so we remove a region from $T_0$ containing 5  vertices of $\Gamma_0$. Moreover, 
  there are exactly  $|b|$   torus in which we  remove a region $R_l$ containing  all the vertices of $\Gamma_l$ except one and (iii) holds. Summing up, if  $b<m$ (resp. $b>M$) then  the number of true vertices of $P_S$ increases by $m-b$ (resp. $b-M$) with respect to the case $m\leq b\leq M$. 
  
  As a consequence $P_S$ has $3(d+r+2h-2)+\sum_{k=1}^{r}(S(p_k/q_k)-2)+f_{m,M}(b)$ true vertices.
\end{remark}

\begin{figure}[h!]                      
\begin{center}                         
\includegraphics[width=12cm]{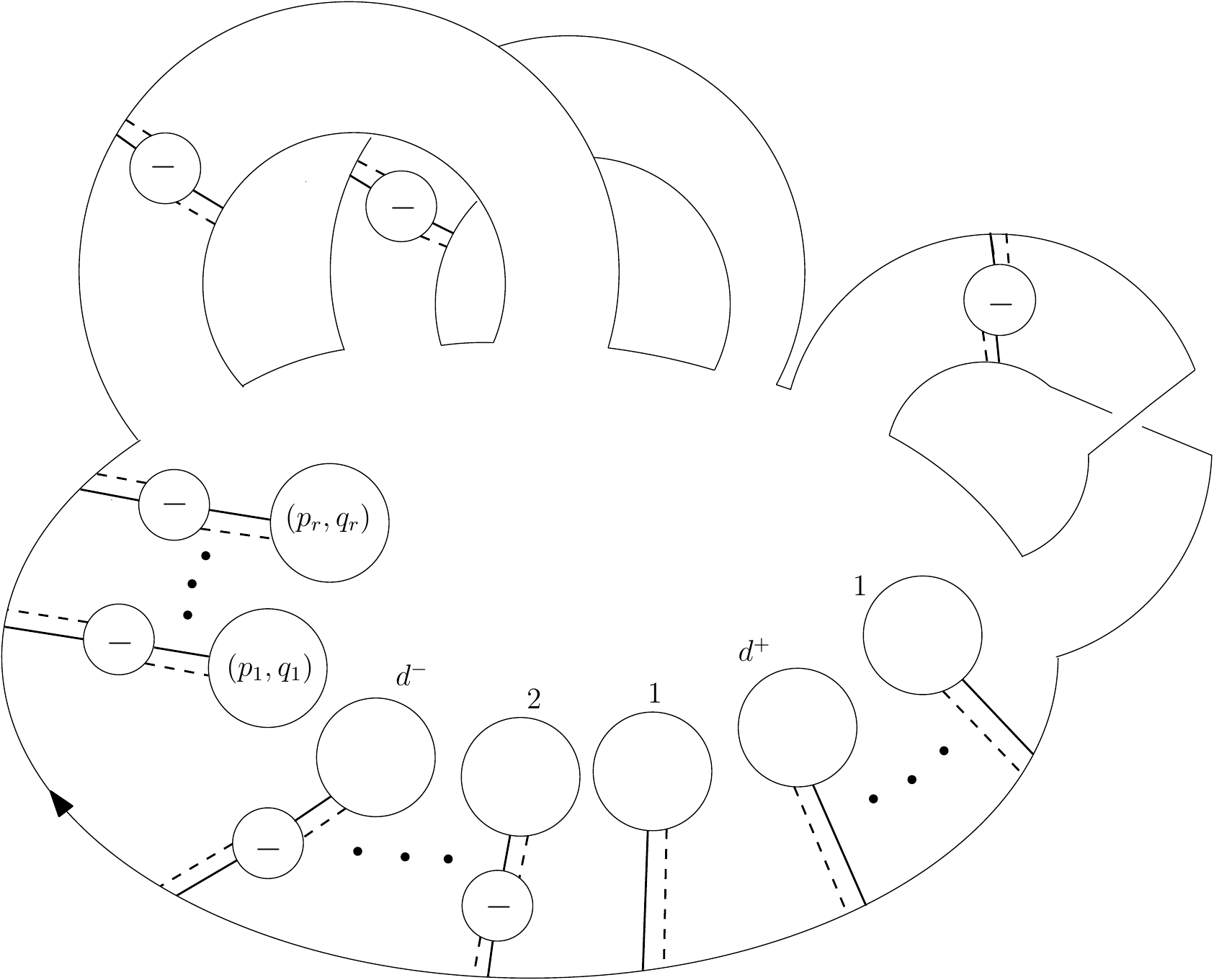}
\caption[legenda elenco figure]{An optimal choice for the shifts corresponding to $(1,-1)$-fibres, when $b=m\leq0$.}\label{negativo}
\end{center}
\end{figure}

\begin{figure}[h!]                      
\begin{center}                         
\includegraphics[width=12cm]{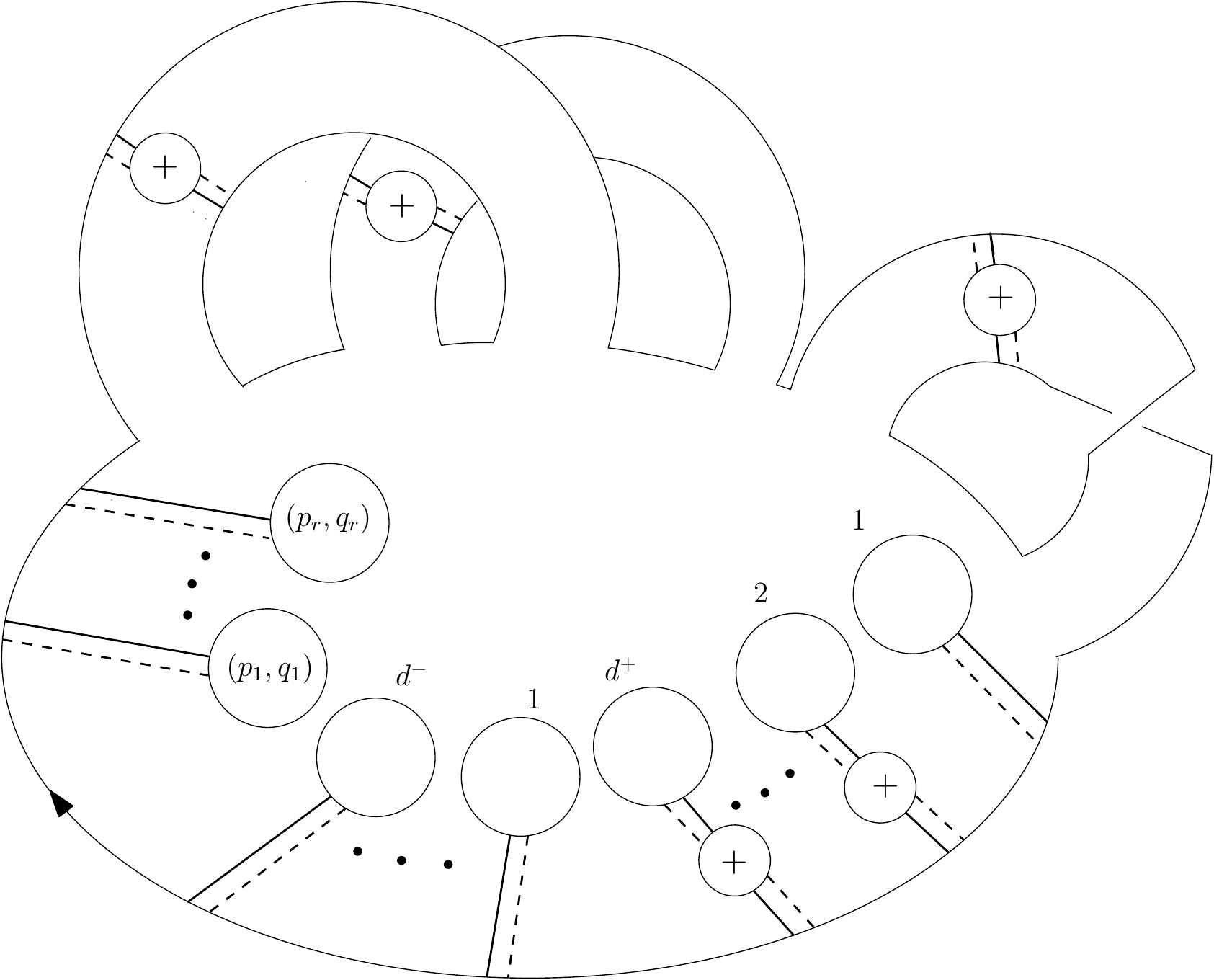}
\caption[legenda elenco figure]{An optimal choice for the shifts corresponding to  $(1,1)$-fibres, when $b=M\geq0$.}\label{positivo}
\end{center}
\end{figure}

\begin{teo} \label{regular} Let $M$ be a graph manifold associated to a decomposition 
graph $G=(V,E,\iota)$ such that  $A_j\neq\pm H$ for any $j\in J$ (i.e., $J'=\emptyset$),  then
\begin{eqnarray*}
 c(M)&\le& 5(|E|-|V|+1)+\sum_{j\in J}(S(\beta_j/\delta_j)-1)+\\
 &\ &+\sum_{i\in I}\left(3(d_i+r_i+2h_i-2)+\sum_{k=1}^{r_i}(S(p_k/q_k)-2)+f_{m_i,M_i}(b_i)\right),
\end{eqnarray*}
where $m_i=-r_i-h_i-d^-_i+1$ and $M_i=h_i+d^+_i-1$.
\end{teo}
\begin{proof}
Let $T=(V,E_{T})$ be a spanning tree of $G$ and consider  the graph manifold  $M_{T}$  (with boundary if $T\ne G$).
 We will construct a skeleton $P_{M_T}$ for $M_{T}$ by assembling    skeletons   of its Seifert pieces (constructed as above) with 
 skeletons of  thickened tori corresponding to edges of $T$. More precisely, for each $j\in J_T$, let $T'_j\subset\partial S_{i'_j}$ and  $T''_j\subset\partial S_{i''_j}$ be the boundaries attached by $A_j$. We construct a skeleton $P_{A_j}$  for  $T''_{j}\times I$, and assemble $P_{S_{i'_j}}$ with $P_{A_j}$ using  the map $A_j:T'_j\to T''_j=T''_j\times\{0\}$ and $P_{A_j}$ with $P_{S_{i''_j}}$ with the identification  $T''_j\times\{1\}= T''_j$.

Given $j\in J_T$, let  $\theta_j$ be the theta graph corresponding to $A_j\tau_-$. We construct the skeleton $P_{A_j}$  by assembling flip blocks (see Figure \ref{flip}) so that 
 (i) $ P_{A_j}\cap (T''_j\times \{0\})=\theta_j$  and (ii) $ P_{A_j}\cap (T''_j\times \{1\})=\theta_+$. By Lemma \ref{matrice} we have $c_{A_j}=d(A_j\tau_-, \tau_+)=S(\beta_j/\delta_j)-1$, so the number of flip blocks required to construct $P_{A_j}$, as well as the number of true vertices of $P_{A_j}$, is  $S(\beta_j/\delta_j)-1$.

By Remark \ref{ottimale}, we can construct the skeleton    $P_{S_i}$  having \linebreak $3(d_i+r_i+2h_i-2)+\sum_{k=1}^{r_i}(S(p_k/q_k)-2)+f_{m_i,M_i}(b_i)$ true vertices. So   $P_{M_T}$  has $$\sum_{j\in J_T}(S(\beta_j/\delta_j)-1)+\sum_{i\in I}\left(3(d_i+r_i+2h_i-2)+\sum_{k=1}^{r_i}(S(p_k/q_k)-2)+f_{m_i,M_i}(b_i)\right)$$ true vertices.

 For each  $j\in J\setminus J_T$, the matrix  $A_j$  identifies two boundary components of $M_{T}$. Denote with $M_{T\cup e_j}$ the resulting manifold. Construct $P_{A_j}$    such  that 
 \hbox{(i) $ P_{A_j}\cap (T''_j\times \{0\})=\theta_j$}  and (ii) $ P_{A_j}\cap(T''_j\times \{1\})=\theta'$, where $\theta'$ is at distance one from $\theta_+$ and is closer to $\theta_j$ than $\theta_+$. The graphs $\theta'$ and $\theta_+$  differ for a flip, since they correspond to adjacent triangles, and therefore,  as shown in Figure \ref{flip3}, we have $i(\theta',\theta_+)=2$, where $i(\cdot,\cdot)$ denotes the geometric intersection (i.e., the minimum number of intersection points up to isotopy).  Consider the polyhedron $P_{M_T}\cup P_{A_j} \cup (T_j''\times \{1\})$:  it is a skeleton for $P_{M_T\cup e_j}$. Since  $P_{A_j}$  consists of $S(\beta_j/\delta_j)-2$ flip blocks and   the graph $\theta_+\cup\theta'$ has 6 vertices of degree greater than 2,  the new polyhedron has $5+ S(\beta_j/\delta_j)-1$ true vertices more than $P_{M_T}$. By repeating this construction for any $j\in J\setminus J_T$ and observing that  $|J\setminus J_T|=|E\setminus E_T| =|E|-|V|+1$, we get the statement.
\end{proof}

\begin{figure}[h!]                      
\begin{center}                         
\includegraphics[width=4cm]{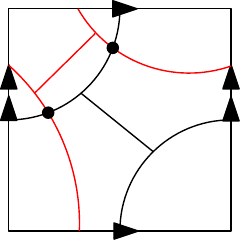}
\caption[legenda elenco figure]{The two intersections of two theta graphs differing for a flip move.}\label{flip3}
\end{center}
\end{figure}

\subsection{An intermediate case}
Now we deal with the class of graph manifolds having a decomposition graph with a spanning tree containing all the edges associated to the matrices $\pm H$. This case seems to be rather technical, but  it is quite interesting since, up to complexity 12, about  $99\%$ of prime  graph manifolds belong to this class (they are exactly  14346 out of 14502). 

Before stating our result we  introduce some notations.
If $J'\ne\emptyset$  let $\Psi=\left\{\psi:J'\to\{+,-\}\right\}$ and set:\footnote{Clearly, $d_i=d^+_i+d^-_i+d^+_{i,\psi}+d^-_{i,\psi}$.}
\begin{itemize}
 \item $d^+_{i,\psi}=\left|\{j\in J'\mid i'_j=i,\psi(j)=+\}\right|+\left|\{j\in J'\mid i''_j=i,\psi(j)=+\}\right|$;
 \item $d^-_{i,\psi}=\left|\{j\in J'\mid i'_j=i,\psi(j)=-\}\right|+\left|\{j\in J'\mid i''_j=i,\psi(j)=-\}\right|$.
\end{itemize}

\begin{teo} \label{tree} Let $M$ be a graph manifold associated to a  decomposition 
graph $G=(V,E,\iota)$. If there exists a spanning tree \hbox{$T\in\mathcal T_G$} of $G$ such that $A_j\neq\pm H$ for any $j\in J\setminus J_T$, 
then
\begin{eqnarray*}
 c(M)&\le& 5(|E|-|V|+1)+\sum_{j\in J''}(S(\beta_j/\delta_j)-1)+\\
  &\ & +\sum_{i\in I}\left(3(d_i+r_i+2h_i-2)+\sum_{k=1}^{r_i}(S(p_k/q_k)-2)\right)+\\
  &\ &+\min_{\psi\in \Psi}\left\{\sum_{i\in I}f_{m_i,M_i}(b_i)\right\},
\end{eqnarray*}
where $m_i=-r_i-h_i-d^-_i-d^-_{i,\psi}+1$ and $M_i=h_i+d^+_i+d^+_{i,\psi}-1$.
\end{teo}

\begin{proof}
As in the  proof of Theorem  \ref{regular}, we start by constructing a skeleton $P_{M_T}$ for the graph manifold   $M_{T}$  (with boundary if $T\ne G$). 
 By Lemma \ref{matrice}  we have $0=c_{\pm H}=d(\pm H\tau_-, \tau_-)=d(\pm H\tau_+, \tau_+)$, so whenever  $j\in J'_T=J'$, no flip block is required in $P_{A_j}$ and  we can  assemble directly  $P_{S_{i'_j}}$ with  $P_{S_{i''_j}}$. So,
 if $T_j'$ and $T_j''$ denote, respectively, the boundary components of  $S_{i'_j}$ and $S_{i''_j}$  glued by $A_j$, we require  that either (i)  $P_{S_{i'_j}}\cap T_j'=\theta_+$  and \hbox{$P_{S_{i''_j}}\cap T_j''=\theta_+$} or (ii) $P_{S_{i'_j}}\cap T_j'=\theta_-$  and $P_{S_{i''_j}}\cap T_j''=\theta_-$. In order to take care of these two possibilities we use a function $\psi:J'\to\{+,-\}$ of the  edges of $E'$.  If the shifts in the construction of $P_{S_{i'_j}}$ and $P_{S_{i''_j}}$ are chosen so that (i) holds (resp. (ii) holds) set $\psi(j)=-$ (resp.  $\psi(j)=+$). Following the construction of   Remark  \ref{ottimale}, with $d^+=d^+_i+d^+_{i,\psi}$ and $d^-=d^-_i+d^-_{i,\psi}$,  we obtain a  skeleton  $P_{S_{i}}$ with $3(d_i+r_i+2h_i-2)+\sum_{k=1}^{r_i}(S(p_k/q_k)-2)+f_{m_i,M_i}(b_i)$ true vertices.  As a consequence, the minimum  number of true vertices of the skeleton $P_{M_T}$ of $M_T$ is $\sum_{j\in J''_T}(S(\beta_j/\delta_j)-1)
   +\sum_{i\in I}\left(3(d_i+r_i+2h_i-2)+\sum_{k=1}^{r_i}(S(p_k/q_k)-2)\right)+$ \linebreak  $+\min_{\psi\in \Psi}\left\{\sum_{i\in I}f_{m_i,M_i}(b_i)\right\}$.
   
  Since all the matrices associated to the edges $e_j\notin E_T$ are different from $\pm H$,   starting from  $P_{M_T}$   we can construct a   spine for $M$ as described in the  proof of Theorem  \ref{regular}. This concludes the proof.
 \end{proof}

\subsection{The general case}

For the general case we need to consider a certain set of spanning trees of $G$.

 Let  $\phi:{\mathcal T}_G\to\N$  be the function defined by $\phi(T)=|(E-E_T)\cap E'|$. A spanning tree 
$T\in {\mathcal T}_G$ is called {\it optimal} if $\phi(T)\le \phi (T')$ for any $T'\in {\cal T}_G$. 
Denote the set of the optimal spanning trees of $G$ with  ${\cal O}_G$  and  
set \hbox{$\Phi(G)=\min\{\phi(T)\mid T\in {\cal T}_G\}$} (i.e., $\Phi(G)=\phi(T)$, for any $T\in {\cal O}_G$).
Finally, let  
$$ \Psi_T=\{\psi:J_T'\to\{+,-\}\},\quad
\Psi'_T=\{\psi':J'\setminus J'_T\to\{++,+,+-,-+,-,--\}\}
$$
and set:
\begin{itemize}
\item $d^+_{i,\psi,T}=\left|\{j\in J'_T\mid i'_j=i,\psi(j)=+\}\right|+\left|\{j\in J'_T\mid i''_j=i,\psi(j)=+\}\right|$;
\item $d^-_{i,\psi,T}=\left|\{j\in J'_T\mid i'_j=i,\psi(j)=-\}\right|+\left|\{j\in J'_T\mid i''_j=i,\psi(j)=-\}\right|$;
\item $d^+_{i,\psi',T}=2\left|\{j\in J'-J'_T\mid i'_j=i,\psi'(j)=++\}\right|+\left|\{j\in J'-J'_T\mid i'_j=i,\psi'(j)=+\}\right|+
  +\left|\{j\in J'-J'_T\mid i''_j=i,\psi'(j)=++\}\right|+2\left|\{j\in J'-J'_T\mid i''_j=i,\psi'(j)=+\}\right|+ +
 \left|\{j\in J'-J'_T\mid i'_j=i,\psi'(j)=+-\}\right| + \left|\{j\in J'-J'_T\mid i''_j=i,\psi'(j)=-+\}\right|$;
 \item $d^-_{i,\psi',T}=2\left|\{j\in J'-J'_T\mid i'_j=i,\psi'(j)=--\}\right|+\left|\{j\in J'-J'_T\mid i'_j=i,\psi'(j)=-\}\right| + 
 +\left|\{j\in J'-J'_T\mid i''_j=i,\psi'(j)=--\}\right|+2\left|\{j\in J'-J'_T\mid i''_j=i,\psi'(j)=-\}\right|+ + \left|\{j\in J'-J'_T\mid i'_j=i,\psi'(j)=-+\}\right| + \left|\{j\in J'-J'_T\mid i''_j=i,\psi'(j)=+-\}\right|$.
\end{itemize}

We are now ready to state the general case.

 \begin{teo}  \label{generale} Let $M$ be a graph manifold associated to a  decomposition graph $G=(V,E,\iota)$. Then
\begin{eqnarray*}
 c(M)&\le& 5(|E|-|V|+1)+\Phi(G)+\sum_{j\in J''}(S(\beta_j/\delta_j)-1)+\\
 &\ & +\sum_{i\in I}\left(3(d_i+r_i+2h_i-2)+\sum_{k=1}^{r_i}(S(p_k/q_k)-2)\right)+\\
 &\ &+ \min_{T\in {\cal O}_G}\left\{ \min_{\psi\in \Psi_T,\psi'\in \Psi'_T}\{\sum_{i\in I}f_{m_i,M_i}(b_i)\}\right\},
\end{eqnarray*}
where $m_i=-r_i-h_i-d^-_i-d^-_{i,\psi,T}-d^-_{i,\psi',T}+1$ and $M_i=h_i+d^+_i+d^+_{i,\psi,T}+d^+_{i,\psi',T}-1$. 
\end{teo}
\begin{proof}
Let $T=(V,E_T)\in\mathcal O_G$. Given  $\psi\in\Psi_T$, we construct a skeleton $P_{M_T}$ for $M_T$ as described in the proof of Theorem \ref{tree}. If $j\in J'\setminus J'_T$, then $A_j=\pm H$ and it glues together  two toric boundary components $T_j'\subset\partial S_{i'_j}$ and $T_j''\subset\partial S_{i''_j}$ of   $M_T$.

Let $(\theta'_j,\theta''_j)=(P_{S_{i'_j}}\cap T'_j,P_{S_{i''_j}}\cap T''_j)$. Since $A_j\Delta_{\pm}=\Delta_{\pm}$ and $i(\theta_+,\theta_+)=i(\theta_-,\theta_+)=i(\theta_-,\theta_-)=2$, we have to consider all  possible cases of \hbox{$\theta_j',\theta_j''\in\{\theta_+,\theta_-\}$}.  If  $(\theta_j',\theta_j'')=(\theta_+,\theta_-)$ (resp.  $(\theta_j',\theta_j'')=(\theta_-,\theta_+)$) then we define $\psi'(j)=-+$ (resp. $\psi'(j)=+-$). If $(\theta_j',\theta_j'')=(\theta_+,\theta_+)$ or   $(\theta_j',\theta_j'')=(\theta_-,\theta_-)$ we can use Remark \ref{regular_matrices} in order to obtain a better estimate for $c(M)$. Indeed, if we replace the matrix $A_j$  with $A_j U^{\mp 1}$ (resp. $U^{\pm}A_j$), then  $b_{i'_j}$ (resp. $b_{i''_j}$) is replaced with $b_{i'_j}\mp 1$ (resp. $b_{i''_j}\mp 1$), but anyway  $i(\theta'_-,\theta_-)=i(\theta''_-,\theta_-)=i(\theta'_+,\theta_+)=i(\theta''_+,\theta_+)=2$, where $\theta'_-,\theta''_-,\theta'_+,\theta_+''$ are the theta graphs associated to the triangles  $A_jU^{-1}\Delta_-, UA_j\Delta_-, A_jU\Delta_+, U^{-1}A_j\Delta_+$, respectively. Indeed, $A_jU^{-1}\Delta_-, UA_j\Delta_-$ are adjacent to $\Delta_-$ and $A_jU\Delta_+, U^{-1}A_j\Delta_+$ are adjacent to $\Delta_+$. More precisely, if we take $(\theta_j',\theta_j'')=(\theta_-,\theta_-)$ and  replace $A_j$ with $A_j U^{-1}$ (resp. $UA_j$) we set   $\psi'(j)=++$ (resp. $\psi'(j)=+$), while if we take \hbox{$(\theta_j',\theta_j'')=(\theta_+,\theta_+)$} and  replace $A_j$ with $A_jU$ (resp. $U^{-1}A_j$) we set $\psi'(j)=--$ (resp. $\psi'(j)=-$). As a result, the skeleton $P_{M_T}$ determined by $\psi\in \Psi_T$ and $\psi'\in\Psi'_T$ has    $$\sum_{j\in J_T''} S(\beta_j/\delta_j)-1)+\sum_{i\in I}\left(3(d_i+r_i+2h_i-2)+\sum_{k=1}^{r_i}(S(p_k/q_k)-2)\right)+\sum_{i\in I}f_{m_i,M_i}(b_i)$$ true vertices.

A spine for $M$ is given by the union of $P_{M_T}$ with: (i) the skeleton $P_{A_j}\cup T_j''\times\{1\}$ described in  the  proof of Theorem \ref{regular} and having $5+(S(\beta_j/\delta_j)-1)$ true vertices   for each $j\in J''\setminus J''_T$, and  (ii) the torus $T_ j''$, containing $6$ true vertices for each $j\in J'\setminus J'_T$. Since  $|J'\setminus J'_T| + |J''\setminus J''_T| = |E\setminus E_T|=|E|-|V|+1$ and $|J'\setminus J'_T|=\phi(G)$ we get the statement. 
\end{proof}

The sharpness of the  previous upper bound  in all  known cases justifies the following:

\begin{con} The upper bound given in Theorem \ref{generale}  is sharp for all closed connected  orientable prime graph manifolds.

\end{con}

\bigskip

\begin{flushleft}


\vbox{

Alessia~CATTABRIGA\\
Department of Mathematics, University of Bologna\\
Piazza di Porta San Donato 5, 40126 Bologna, ITALY\\
e-mail: \texttt{alessia.cattabriga@unibo.it}\\}

\bigskip

Michele~MULAZZANI\\
Department of Mathematics and ARCES, University of Bologna\\
Piazza di Porta San Donato 5, 40126 Bologna, ITALY\\
e-mail: \texttt{michele.mulazzani@unibo.it}\\

\end{flushleft}

\end{document}